\documentclass[final,3p,times]{elsarticle2}
\usepackage{graph}

\usepackage{amsthm,amsmath,amssymb}
\usepackage{enumitem}
\usepackage{graphicx}

\usepackage[colorlinks=true,citecolor=black,linkcolor=black,urlcolor=blue]{hyperref}

\newcommand{\com}[1]{}
\theoremstyle{plain}
\newtheorem{theorem}{Theorem}
\newtheorem{lemma}[theorem]{Lemma}
\newtheorem{corollary}[theorem]{Corollary}

\newtheorem{conjecture}[theorem]{Conjecture}

\newtheorem{observation}[theorem]{Observation}
\newtheorem{claim}[theorem]{Claim}

\theoremstyle{definition}
\newtheorem{definition}[theorem]{Definition}

\theoremstyle{remark}

\begin{document}
\begin{frontmatter}
\title{On Induced Colourful Paths in Triangle-free Graphs}
\author{Jasine Babu\footnote{Indian Institute of Technology, Palakkad, India. E-mail: \texttt{jasine@iitpkd.ac.in}}\hspace{.5in} Manu Basavaraju\footnote{National Institute of Technology Karnataka, Surathkal, India. E-mail: \texttt{manub@nitk.ac.in}}\hspace{.5in} L.~Sunil Chandran\footnote{Indian Institute of Science, Bangalore, India. E-mail: \texttt{sunil@iisc.ac.in}}\hspace{.5in} Mathew~C. Francis\footnote{Indian Statistical Institute, Chennai, India. E-mail: \texttt{mathew@isichennai.res.in}}}

\begin{abstract}
Given a graph $G=(V,E)$ whose vertices have been properly coloured, we say that a path in $G$ is \textit{colourful} if no two vertices in the path have the same colour. It is a corollary of the Gallai-Roy-Vitaver Theorem that every properly coloured graph contains a colourful path on $\chi(G)$ vertices. We explore a conjecture that states that every properly coloured triangle-free graph $G$ contains an induced colourful path on $\chi(G)$ vertices and prove its correctness when the girth of $G$ is at least $\chi(G)$. Recent work on this conjecture by Gy\'arf\'as and S\'ark\"ozy, and Scott and Seymour has shown the existence of a function $f$ such that if $\chi(G)\geq f(k)$, then an induced colourful path on $k$ vertices is guaranteed to exist in any properly coloured triangle-free graph $G$.
\end{abstract}
\end{frontmatter}
\section{Introduction}
All graphs considered in this paper are simple, undirected and finite.
For a graph $G=(V,E)$, we denote the vertex set of $G$ by $V(G)$ and the edge set of $G$ by $E(G)$. 
A function $c:V(G)\rightarrow\{1,2,\ldots,k\}$ is said to be a \emph{proper $k$-colouring} of $G$ if for any edge $uv\in E(G)$, we have $c(u)\neq c(v)$. A graph is \textit{properly coloured}, if it has an associated proper $k$-colouring $c$ specified (for some $k$). The minimum integer $k$ for which a graph $G$ has a proper $k$-colouring is the \emph{chromatic number} of $G$, denoted by $\chi(G)$. 
A subgraph $H$ of a properly coloured graph $G$ is said to be \emph{colourful} if no two vertices of $H$ have the same colour. If a colourful subgraph $H$ of $G$ is also an induced subgraph, then we say that $H$ is an \emph{induced} colourful subgraph of $G$. The \emph{length} of a path or cycle in $G$ is the number of edges in the path or cycle. Therefore, a path on $t$ vertices has length $t-1$ and a cycle on $t$ vertices has length $t$.

The following classical result of Gallai, Roy and Vitaver (see~\cite{West})
tells us that that every (not necessarily optimally) properly coloured graph $G$ has a colourful path on $\chi(G)$ vertices (an alternative proof for this is given in Theorem~\ref{thm:colpath}).

\begin{theorem}[Gallai-Roy-Vitaver]
Let $G$ be a graph and let $H$ be any directed graph obtained by orienting the edges of $G$. Then $H$ contains a directed path on $\chi(G)$ vertices.
\end{theorem}

Indeed, given a properly coloured graph $G$, we can construct a directed graph $H$ by fixing an arbitrary order on the colours and orienting every edge from the vertex of lower colour to the vertex of higher colour. Then, by the Gallai-Roy-Vitaver Theorem, we know that there is directed path on $\chi(G)$ vertices in $H$ and this is a colourful path in $G$ as the colours on this path are strictly increasing.

We are interested in the question of when one can find colourful paths on $\chi(G)$ vertices that are also induced in a given properly coloured graph $G$. Note that the colourful path on $\chi(G)$ vertices that should exist 
in any properly coloured graph $G$, as noted above, may not always be an induced path. In fact, when $G$ is a complete graph, there is no 
induced path on more than two vertices in the graph. The following conjecture is due to N.~R.~Aravind~\cite{Aravind}.

\begin{conjecture}\label{conj:main}
Let $G$ be a triangle-free graph that is properly coloured. Then there is an induced colourful path on $\chi(G)$ vertices in $G$. 
\end{conjecture}

Recently, Gy\'arf\'as and S\'ark\"ozy~\cite{GyarfasSarkozy} studied this conjecture and showed that there exists a function $f(k)$ such that in any properly coloured graph $G$ with girth at least 5 and $\chi(G)\geq f(k)$, there is an induced colourful path on $k$ vertices. This was improved by Scott and Seymour~\cite{ScottSeymour}, who removed the girth condition, and showed that for any two integers $k$ and $t$, there exists a function $f(k,t)$ such that in any properly coloured graph $G$ with $\omega(G)\leq t$ and $\chi(G)\geq f(k,t)$, there is an induced colourful path on $k$ vertices (here, $\omega(G)$ denotes the maximum size of a clique in $G$). 

A necessary condition for Conjecture~\ref{conj:main} to hold is the presence of an induced path on $\chi(G)$ vertices in any triangle-free graph $G$. 
Indeed something stronger is known to be true: each vertex in a triangle-free graph $G$ is the starting point of an induced path on $\chi(G)$ vertices~\cite{Gyarfas}. Concerning induced trees, Gy\'{a}rf\'{a}s~\cite{Gyarfas1} and Sumner~\cite{Sumner} conjectured that there exists an integer-valued function $f$ defined on finite trees with the property that 
every triangle-free graph $G$ with $\chi(G)=f(T)$ contains $T$ as an induced subgraph. This was proven true for trees of radius two by 
Gy\'{a}rf\'{a}s, Szemer\'{e}di, and Tuza~\cite{GyarSzTu}. 
There have been several investigations on variants of the Gallai-Roy-Vitaver Theorem~\cite{AlishahiTT11, Li}. 
Every connected graph $G$ other than $C_7$ admits a proper $\chi(G)$-colouring such that every vertex of $G$ is the beginning of 
a (not necessarily induced) colourful path on $\chi(G)-1$ vertices~\cite{AlishahiTT11}.
A stronger version of the Gallai-Roy-Vitaver Theorem that guarantees an induced directed path on $\chi(G)$ vertices in any directed graph $G$ would have easily implied Conjecture~\ref{conj:main}. Clearly, such a theorem cannot be true for every directed graph. But Kierstead and Trotter~\cite{KiersteadT92} show that no such result can be obtained even if the underlying undirected graph of $G$ is triangle-free. They show that for every natural number $k$, there exists a digraph $G$ such that its underlying undirected graph is triangle-free and has chromatic number $k$, but $G$ has no induced directed path on 4 vertices.

Conjecture~\ref{conj:main} is readily seen to be true for any triangle-free graph $G$ with $\chi(G)=3$, because the colourful path guaranteed to exist in $G$ by the Gallai-Roy-Vitaver Theorem is also an induced path in $G$. For the same reason, the conjecture is true for any graph $G$ with $g(G)>\chi(G)$, where $g(G)$ is the \emph{girth} of $G$, or the length of the shortest cycle in $G$.
In this paper, we first prove Conjecture~\ref{conj:main} for the case when $\chi(G)=4$. Note that it follows from the above observation that to prove this, we only need to consider graphs $G$ with $g(G)=4$. Proving Conjecture~\ref{conj:main} in its full generality even for the case when $\chi(G)=5$ does not seem to be easy. As explained above, in order to prove the conjecture for the case $\chi(G)=5$, we only need to prove it for graphs with $g(G)\in\{4,5\}$. Our approach shows that the conjecture is true when $g(G)=5$; the case when $g(G)=4$ is open. Scott and Seymour~\cite{ScottSeymour} mention that they verified by hand that the conjecture holds for all possible colourings of the Mycielski graph on 23 vertices having chromatic number 5. One natural way to weaken the conjecture would be to restrict the girth of the graph to be above some constant fraction of $\chi(G)$. Even this appears to be difficult. The main result of this paper shows that for each value of $\chi(G)\geq 4$, the conjecture is true for graphs with $g(G)=\chi(G)$.
\section{Preliminaries}
Notation used in this paper is the standard notation used in graph theory (see e.g.~\cite{West}). We shall now describe a special greedy colouring procedure for an already coloured graph that will later help us in proving our main result.
\paragraph{\textbf{The refined greedy algorithm}}Given a properly coloured graph $G$ with the colouring $\beta$, we will construct a new proper colouring $\alpha:V(G)\rightarrow \mathbb{N}^{>0}$ of $G$, using the algorithm given below. 
Let $b_1<b_2<\cdots<b_t$ be the colours used by $\beta$.  
\begin{tabbing}
~~~~~\=~~~~~\=~~\kill
For every vertex $v\in V(G)$, set $\alpha(v)\leftarrow 0$\\
\textbf{for} $i$ from 1 to $t$ \textbf{do}\\
\>\textbf{for} vertex $v$ with $\beta(v)=b_i$ and $\alpha(v)=0$ \textbf{do}\\
\>\>Colour $v$ with the least positive integer that has \\ 
\>\>not already been assigned to a neighbour of it\\
\>\>i.e., set $\alpha(v)\leftarrow\min (\mathbb{N}^{>0} \setminus \{\alpha(u)\colon u\in N(v)\})$.
\end{tabbing}

Let $G$ be a graph with a proper colouring $\beta$ and let $\alpha$ be the proper colouring that is constructed by the refined greedy algorithm. We now define a \emph{decreasing path} in $G$ as follows:  

\begin{definition}[Decreasing path]
A path $u_1u_2\ldots u_l$ in $G$ is said to be a ``decreasing path'' if for $2\leq i\leq l$, $\alpha(u_i)<\alpha(u_{i-1})$ and $\beta(u_i)<\beta(u_{i-1})$.
\end{definition}

\begin{lemma}\label{lem:path}
Let $v\in V(G)$ and $X=\{a_1,a_2,\ldots,a_{|X|}\}\subseteq\{1,2,\ldots,\alpha(v)-1\}$ such that $a_1<a_2<\cdots<a_{|X|}$. Then there is a decreasing path $vu_{|X|}u_{|X|-1} \ldots u_1$ in $G$ such that for $1\leq i\leq |X|$, $\alpha(u_i)=a_i$.
\end{lemma}
\begin{proof}
We shall prove this by induction on $\alpha(v)$. It is easy to see that the statement is true for the base case when $\alpha(v)=1$ (because $X=\emptyset$ in that case). Suppose that the statement is true for vertices $u$ with $\alpha(u)<\alpha(v)$. Note that the refined greedy algorithm colours each vertex exactly once. The fact that the algorithm assigned the colour $\alpha(v)$ to $v$ implies that at the time of colouring $v$, we had $\alpha(v)=\min(\mathbb{N}^{>0}\setminus \{\alpha(u)\colon u\in N(v)\})$. Since $a_{|X|}<\alpha(v)$, this means that at that point of time, there was $w\in N(v)$ having $\alpha(w)=a_{|X|}$. Also, since $w$ had previously been coloured by the algorithm, we have $\beta(w)\leq\beta(v)$. But as $w\in N(v)$, we know that $\beta(w)\neq\beta(v)$, giving us $\beta(w)<\beta(v)$. Now, applying the induction hypothesis on $w$ and the set $X\setminus\{a_{|X|}\}$, we get that there is a decreasing path $wu'_{|X|-1}u'_{|X|-2}\ldots u'_1$ in $G$ such that for $1\leq i\leq |X|-1$, $\alpha(u'_i)=a_i$. It is clear that $vwu'_{|X|-1}u'_{|X|-2}\ldots u'_1$ is then a decreasing path of the form $vu_{|X|}u_{|X|-1}\ldots u_1$, where for $1\leq i\leq |X|$, $\alpha(u_i)=a_i$.
\end{proof}

The above observation about the refined greedy algorithm can be used to show that there is a colourful path on $\chi(G)$ vertices in every properly coloured graph $G$ (without using the Gallai-Roy-Vitaver Theorem).

\begin{theorem}\label{thm:colpath}
If $G$ is any graph whose vertices are properly coloured, then there is a colourful path on $\chi(G)$ vertices in $G$.
\end{theorem}
\begin{proof}
Let $\beta$ denote the proper colouring of $G$.
Run the refined greedy algorithm on $G$ to generate the colouring $\alpha$. Clearly, the algorithm will use at least $\chi(G)$ colours as the colouring $\alpha$ generated by the algorithm is also a proper colouring of $G$. Let $v$ be any vertex in $G$ with $\alpha(v)\geq\chi(G)$. Now consider the set $X=\{1,2,\ldots,\chi(G)-1\}$. By applying Lemma~\ref{lem:path} on $v$ and $X$, we can conclude that there is a path on $\chi(G)$ vertices starting at $v$ on which the colours in the colouring $\beta$ are strictly decreasing. This path is a colourful path on $\chi(G)$ vertices in $G$.
\end{proof}

\begin{corollary}\label{cor:colpath}
Any properly coloured graph $G$ with $g(G)>\chi(G)$ has an induced colourful path on $\chi(G)$ vertices.
\end{corollary}
\begin{proof}
If $g(G)>\chi(G)$, then the colourful path given by Theorem~\ref{thm:colpath} is an induced path in $G$.
\end{proof}

This implies that the conjecture is true for all triangle-free graphs with chromatic number at most 3.
It also implies that in order to prove Conjecture~\ref{conj:main}, one only has to consider graphs $G$ with $4\leq g(G)\leq\chi(G)$.
The main result of this paper is that in any properly coloured graph $G$ with $4\leq g(G)=\chi(G)$, there exists an induced colourful path on $\chi(G)$ vertices.

\section{Induced colourful paths in graphs with girth equal to chromatic number}
In this section, we shall prove our main result, given by the theorem below.
\begin{theorem}\label{thm:main}
Let $G$ be a graph with $g(G)\geq\chi(G)=k$, where $k\geq 4$, and whose vertices have been properly coloured. Then there exists an induced colourful path on $k$ vertices in $G$.
\end{theorem}
We can assume that $G$ is ``critical'', i.e., every proper induced subgraph of $G$ has chromatic number less than $k=\chi(G)$. This is because if the theorem is proven for critical graphs, then since $G$ contains a critical induced subgraph $G'$ having $\chi(G')=k$, we can apply the theorem to $G'$, whose vertices are coloured with same colours that they had in $G$, to get an induced colourful path in $G'$ containing $k$ vertices (note that $g(G')\geq g(G)$). This path is clearly also an induced colourful path in $G$. Assuming that $G$ is critical, the following observation is not too hard to see~\cite{West}.

\begin{observation}\label{obs:mindegree}
Every vertex in $G$ has degree at least $k-1$.
\end{observation}

Note that critical graphs are connected, and hence we assume that $G$ is connected.
Note also that by Corollary~\ref{cor:colpath}, we can assume $g(G)=k$.
Let $\beta:V(G)\rightarrow\mathbb{N}^{>0}$ denote the proper colouring of $G$ that is given.

A $k$-cycle in $G$ in which no colour repeats is said to be a \emph{colourful $k$-cycle}, sometimes shortened to just ``colourful cycle''. Notice that every colourful cycle in $G$ is also an induced cycle as $g(G)=k$. Also, from here onwards, we shorten ``colourful path on $k$ vertices'' to just ``colourful path''.

Suppose that there is no induced colourful path on $k$ vertices in $G$.

\begin{observation}\label{obs:colpath}
Since $g(G)=k$, if $y_1y_2\ldots y_k$ is a colourful path on $k$ vertices in $G$, then the edge $y_1y_k\in E(G)$. Thus, $y_1y_2\ldots y_ky_1$ is a colourful $k$-cycle in $G$.
\end{observation}

Let $\alpha$ be a proper colouring of $G$ generated by running the refined greedy algorithm on $G$. We shall refer to the colours of the colouring $\alpha$ as ``labels''. From here onwards, we shall reserve the word ``colour'' to refer to a colour in the colouring $\beta$. As before, whenever we say that a path or a cycle is ``colourful'', we are actually saying that it is colourful in the colouring $\beta$.

We say that a path with no repeating colours is an ``almost decreasing path'' if the subpath induced by the vertices other than the starting vertex is a decreasing path. Note that any decreasing path is also an almost decreasing path.

\begin{definition}
We say that a set of vertices or a subgraph in $G$ ``sees'' the colour $i$ if one of the vertices in it has colour~$i$.
\end{definition}

If $G_1$ and $G_2$ are two subgraphs of $G$, then we define $G_1\cup G_2$ to be the subgraph of $G$ with vertex set $V(G_1)\cup V(G_2)$ and edge set $E(G_1)\cup E(G_2)$. In particular, if $G_1$ is a subgraph of $G$ and $xy\in E(G)$, we denote by $G_1\cup xy$ the subgraph with vertex set $V(G_1)\cup\{x,y\}$ and edge set $E(G_1)\cup\{xy\}$.

The proof of Theorem~\ref{thm:main} is split into two cases: when $k=4$ and when $k>4$.
\vspace{-0.05in}

\subsection{Case when $k=4$}

In this case, we have $\chi(G)=g(G)=4$.

As $\alpha$ is also a proper colouring of $G$, we know that there exists a vertex $v$ in $G$ with label 4 (i.e., $\alpha(v)=4$).
By Lemma~\ref{lem:path}, there exists a decreasing path $v_4v_3v_2v_1$ where $v_4=v$ and for $1\leq i\leq 3$, we have $\beta(v_i)<\beta(v_{i+1})$ and $\alpha(v_i)=i$. Clearly, as $v_4v_3v_2v_1$ is a decreasing and hence colourful path, by Observation~\ref{obs:colpath}, we have $v_1v_4\in E(G)$. Again by Lemma~\ref{lem:path}, we have a path $vv'_2v'_1$ in which we have $\beta(v'_1)<\beta(v'_2)<\beta(v)$, $\alpha(v'_2)=2$ and $\alpha(v'_1)=1$. Note that $v'_2\neq v_2$ and $v'_1\neq v_1$ (as otherwise $vv'_2v_1v$ would be a triangle in $G$). This means that the vertices in $\{v_4,v_3,v_2,v_1,v'_2,v'_1\}$ are all pairwise distinct. Let $\beta(v_i)=b_i$ for each $i$, where $1 \le i \le 4$. We shall call the colours $b_1,b_2,b_3,b_4$ ``primary colours''. Clearly, we have $b_1<b_2<b_3<b_4$.

\begin{claim}
$\beta(v'_2)=b_2$ and $\beta(v'_1)=b_1$.
\end{claim}
\begin{proof}
Suppose that $\beta(v'_2) \neq b_2$. Then we have that either the path  $v'_2v_4v_3v_2$ or the path $v'_2v_4v_1v_2$ is colourful, which implies that $v'_2v_2 \in E(G)$, a contradiction since $\alpha(v'_2)=\alpha(v_2)$. Therefore we have $\beta(v'_2)=b_2$. Since $vv'_2v'_1$ is a decreasing path, this tells us that $\beta(v'_1)<b_2$. Thus, if $\beta(v'_1)\neq b_1$, the path $v'_1v'_2v_4v_1$ is colourful, implying that $v'_1v_1 \in E(G)$, which is a contradiction since $\alpha(v'_1)=\alpha(v_1)$. We can thus conclude that $\beta(v'_1)=b_1$.
\end{proof}

Now notice that the path $v'_1v'_2v_4v_3$ is colourful and hence we have that $v'_1v_3 \in E(G)$.  We call the vertices in the set $\{v_4,v_3,v_2,v_1,v'_2,v'_1\}$ ``forced vertices''. The vertices of $G$ other than these are called ``optional vertices''. Figure~\ref{fig:4} shows the subgraph induced in $G$ by the forced vertices. From our previous observations, we have that for any forced vertex $w$, $\beta(w)=b_{\alpha(w)}$. The following two observations about forced vertices are easy to verify.

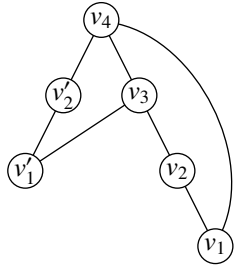
\begin{figure}
\renewcommand{\defradius}{0.23}
\renewcommand{\vertexset}{(v4,5,3),(v3,5.5,2),(v2,6,1),(v1,6.5,0),(v'2,4.5,2),(v'1,4,1)}
\renewcommand{\edgeset}{(v4,v3),(v3,v2),(v2,v1),(v4,v'2),(v'2,v'1),(v4,v1,,,1),(v'1,v3)}
\begin{center}
\begin{tikzpicture}
\drawgraph
\node at (\xy{v4}) {$v_4$};
\node at (\xy{v3}) {$v_3$};
\node at (\xy{v2}) {$v_2$};
\node at (\xy{v1}) {$v_1$};
\node at (\xy{v'2}) {$v'_2$};
\node at (\xy{v'1}) {$v'_1$};
\end{tikzpicture}
\end{center}
\caption{The forced vertices when $k=4$.}\label{fig:4}
\end{figure}

\begin{observation}\label{obs:3path}
Let $u$ be a forced vertex. For $X\subseteq\{b_1,b_2,b_3,b_4\}\setminus\{\beta(u)\}$ such that $|X|=2$, there exists a colourful path on 3 vertices starting at $u$ that sees exactly the colours in $X\cup\{\beta(u)\}$ and contains only forced vertices.
\end{observation}

\begin{observation}\label{obs:nonadj}
Let $x,y$ be forced vertices such that $xy\notin E(G)$ and $\beta(x)\neq\beta(y)$. Then either there exists a forced vertex $x'\in N(x)$ such that $\beta(x')=\beta(y)$ or there exists a forced vertex $y'\in N(y)$ such that $\beta(y')=\beta(x)$.
\end{observation}

\begin{lemma}\label{lem:optionalprops}
Let $u$ be an optional vertex. If $N(u)$ contains a forced vertex $x$, then there exist forced vertices $y,z$ such that $uxzyu$ is a colourful cycle.
\end{lemma}
\begin{proof}
Choose some $X\subseteq\{b_1,b_2,b_3,b_4\}\setminus\{\beta(u),\beta(x)\}$ such that $|X|=2$. From Observation~\ref{obs:3path}, there exists a colourful path $xzy$, where $z$ and $y$ are forced vertices, that sees exactly the colours in $X\cup\{\beta(x)\}$. Clearly, $uxzy$ is a colourful path, which implies by Observation~\ref{obs:colpath} that $uxzyu$ is a colourful cycle.
\end{proof}
\begin{lemma}\label{lem:optforcedneigh}
Every optional vertex is adjacent to at least one forced vertex.
\end{lemma}
\begin{proof}
Consider the set of all optional vertices that have no forced vertices as neighbours. For the sake of contradiction, assume that this set is nonempty. Let $w$ be a vertex in this set that is closest to a forced vertex. As $G$ is connected, $w$ has a neighbour $w'$ such that $w'$ is an optional vertex and $N(w')$ contains a forced vertex. From Lemma~\ref{lem:optionalprops}, there is a colourful cycle $w'xzyw'$, where $x,z,y$ are all forced vertices. If $\beta(w)\neq\beta(z)$, then at least one of the paths $ww'xz$ or $ww'yz$ will be a colourful path, and by Observation~\ref{obs:colpath} we will have $wz\in E(G)$, which contradicts the fact that there was no forced vertex in $N(w)$. We can therefore assume that $\beta(w)=\beta(z)$. Notice that $x$ and $y$ are two nonadjacent forced vertices with $\beta(x)\neq\beta(y)$. By Observation~\ref{obs:nonadj}, we then get that either $N(x)$ contains a forced vertex $x'$ such that $\beta(x')=\beta(y)$ or $N(y)$ contains a vertex forced vertex $y'$ such that $\beta(y')=\beta(x)$. In the former case, $ww'xx'$ is a colourful path and in the latter case, $ww'yy'$ is a colourful path. By Observation~\ref{obs:colpath}, we now have that either $wx'\in E(G)$ or $wy'\in E(G)$. This again contradicts the fact that there are no forced vertices in $N(w)$.
\end{proof}
Let $S_1$ denote the set of optional vertices adjacent to at least one of the forced vertices $\{v_4,v_2,v'_1\}$ and let $S_2$ denote the set of optional vertices adjacent to at least one of the forced vertices $\{v_3,v_1,v'_2\}$.

\begin{lemma}\label{lem:s1s2}
\begin{enumerate}
\itemsep 0in
\renewcommand{\theenumi}{\textit{(\roman{enumi})}}
\renewcommand{\labelenumi}{\textit{(\roman{enumi})}}
\item\label{lem:s1s2disjoint} $S_1$ and $S_2$ are disjoint, and
\item\label{lem:s1s2independent} $S_1$ and $S_2$ are both independent sets.
\end{enumerate}
\end{lemma}
\begin{proof}
First let us show that $S_1$ and $S_2$ are disjoint. Suppose that there is a vertex $w\in S_1\cap S_2$.
Consider any two forced vertices $x$ and $y$ in $N(w)$ such that
$x\in \{v_4,v_2,v'_1\}$ and $y\in \{v_3,v_1,v'_2\}$. As $G$ is triangle-free, we only have the two possibilities $(x=v'_1, y=v_1)$ or $(x=v_2, y=v'_2)$.  Note that this implies that $|N(w)\cap\{v_4,v_2,v'_1\}|=|N(w)\cap\{v_3,v_1,v'_2\}|=1$. This lets us conclude that the set of forced vertices in $N(w)$ is either $\{v'_1,v_1\}$ or $\{v_2,v'_2\}$. We now have that the two forced vertices in the neighbourhood of $w$ have the same colour. Then there cannot exist a colourful cycle containing $w$ in which all the other vertices are forced vertices, contradicting Lemma~\ref{lem:optionalprops}.
This proves~\ref{lem:s1s2disjoint}.

From~\ref{lem:s1s2disjoint}, we have that for each vertex $w\in S_1$, the forced vertices in $N(w)$ all lie in $\{v_4,v_2,v'_1\}$ and for each vertex $w'\in S_2$, the forced vertices in $N(w')$ all lie in $\{v_3,v_1,v'_2\}$. Since we know from Lemma~\ref{lem:optionalprops} and Lemma~\ref{lem:optforcedneigh} that each vertex in $S_1\cup S_2$ has at least two forced vertices in their neighbourhood, we can conclude that each vertex in $S_1$ has at least two neighbours from $\{v_4,v_2,v'_1\}$ and that each vertex in $S_2$ has at least two neighbours from $\{v_3,v_1,v'_2\}$.
This means that for any two $w,w'\in S_1$, there is at least one vertex in $\{v_4,v_2,v'_1\}$ that is a neighbour of both $w$ and $w'$. As $G$ is triangle-free, we can conclude that $ww'\notin E(G)$. For the same reason, for any two vertices $w,w'\in S_2$, we have $ww'\notin E(G)$. This proves~\ref{lem:s1s2independent}.
\end{proof}

From Lemma~\ref{lem:s1s2}\ref{lem:s1s2disjoint}, we know that there are no edges between $S_1$ and $\{v_3$, $v_1,v'_2\}$. Similarly, there are no edges between $S_2$ and $\{v_4,v_2,v'_1\}$. Now, by Lemma~\ref{lem:s1s2}\ref{lem:s1s2independent}, we have that
$S_1 \cup\{v_3,v_1,v'_2\}$ is an independent set and $S_2 \cup \{v_4,v_2,v'_1\}$ is an independent set. Since from Lemma~\ref{lem:optforcedneigh}, we know that $V(G)=S_1\cup S_2\cup\{v_4,v_3,v_2,v_1,v'_2,v'_1\}$, this tells us that $G$ is bipartite, which contradicts the assumption that $\chi(G)=4$. Therefore, there can be no properly coloured graph $G$ such that $g(G)=\chi(G)=4$ with no induced colourful path on 4 vertices. This completes the proof of Theorem~\ref{thm:main} for the case $k=4$.

%%%%%%%%%%%%%%%%%%%%%%%%%%%%%%%%%%%%%%%%%%%%%%%%%%%%%%%%%%%%%%%%%%%%%%%%%%%%%%%%%%%%%%%%%%%%%%    

\subsection{Case when $k>4$}

We now use the the fact that $g(G)=k>4$ to complete the proof. In this case we define forced vertices, primary colours and the primary cycle in a more general way. First, we give a useful lemma.

\begin{lemma}\label{lem:colourfulcycle}
Let $y_1y_2\ldots y_ky_1$ be a colourful $k$-cycle.
Let $z\in N(y_i)\setminus \{y_{i-1},y_{i+1}\}$ for some $i\in\{1,2,\ldots,k\}$. Then $\beta(z)\in \{\beta(y_1),\ldots,\beta(y_k)\}\setminus \{\beta(y_{i-1}),\beta(y_i),$ $\beta(y_{i+1})\}$. (Here we assume that $y_{i+1}=y_1$ when $i=k$ and that $y_{i-1}=y_k$ when $i=1$.)
\end{lemma}
\begin{proof}
Clearly, $z\not\in\{y_1,y_2,\ldots,y_k\}$ as every colourful cycle is an induced cycle. Suppose $\beta(z)\notin \{\beta(y_1),\ldots,\beta(y_k)\}\setminus \{\beta(y_{i-1}),\beta(y_i),\beta(y_{i+1})\}$. Clearly, $\beta(z)\neq\beta(y_i)$. Suppose that $\beta(z) \neq \beta(y_{i+1})$. Then observe that $zy_iy_{i+1}\ldots y_ky_1\ldots y_{i-2}$ is a colourful path on $k$ vertices and hence $zy_{i-2}\in E(G)$. This implies that $zy_iy_{i-1}y_{i-2}z$ is a 4-cycle in $G$, which is a contradiction to the fact that $g(G)=k>4$. Therefore, $\beta(z)=\beta(y_{i+1})$. Then the path $zy_iy_{i-1}\ldots y_1y_k\ldots y_{i+2}$ is a colourful path and the same reasoning as above tells us that there is a 4-cycle $zy_iy_{i+1}y_{i+2}z$ in $G$, which is again a contradiction.
\end{proof}
\begin{corollary}\label{cor:colourfulcycle}
Let $y_1y_2\ldots y_ky_1$ be a colourful $k$-cycle. Let $z\in N(y_i)$ for some $i\in\{1,2,\ldots,k\}$. Then $\beta(z)\in\{\beta(y_1),\ldots,\beta(y_k)\}$.
\end{corollary}
\medskip

\noindent\textbf{The vertex \textit{v}:} Fix $v$ to be a vertex which has the largest label. Since $\alpha$ is also a proper vertex colouring of $G$, it should use at least $k$ labels. In other words, $\alpha(v)\geq k$. (For the proof to go through, we could have chosen any vertex with label $k$ as $v$. But as Lemma~\ref{lem:largestlabel} shows, the vertex with largest label will have label $k$.)
\medskip

\noindent\textbf{Primary cycle:}
By applying Lemma~\ref{lem:path} to $v$ and the set of labels $\{1,2,\ldots,k-1\}$, we can conclude that there exists a decreasing path $v_kv_{k-1}\ldots v_1$ where $v_k=v$ and such that $\alpha(v_i)=i$ for all $1\leq i<k$ and $\beta(v_i)<\beta(v_{i+1})$ for all $1\leq i\leq k-1$. Since this path is colourful, by Observation~\ref{obs:colpath}, $vv_{k-1}v_{k-2}\ldots v_1v$ is a colourful cycle, which we shall call the ``primary cycle''. For $1\leq i\leq k$, we shall denote by $b_i$ the colour $\beta(v_i)$. The set of colours $\{b_k,b_{k-1},\ldots,b_1\}$ shall be called the set of ``primary colours''.
\medskip

\noindent\textbf{Forced vertices:} A vertex $u\in V(G)$ is said to be a ``forced vertex'' if there is a decreasing path from $v$ to $u$. Note that every vertex on the primary cycle is a forced vertex.

\begin{lemma}\label{lem:largestlabel}
$\alpha(v)=k$. Hence, for $1\leq i\leq k$, $\alpha(v_i)=i$.
\end{lemma}
\begin{proof}
Suppose for the sake of contradiction that $\alpha(v)>k$.
By Lemma~\ref{lem:path}, there exists a decreasing path $y_{k+1}y_k\ldots y_1$ where $y_{k+1}=v$ and for $1\leq i\leq k$, we have $\alpha(y_i)=i$ and $\beta(y_i)<\beta(y_{i+1})$.
As the paths $y_{k+1}y_k\ldots y_2$ and $y_ky_{k-1}\ldots y_1$ are both colourful, it must be the case that $y_{k+1}y_2,y_ky_1\in E(G)$. But then, $y_{k+1}y_2y_1y_ky_{k+1}$ is a cycle on four vertices in $G$, which is a contradiction to the fact that $g(G)=k>4$.
\end{proof}

\begin{lemma}\label{lem:uis}
For each $i\in\{1,2,\ldots,k-1\}$ there is exactly one vertex $u_i$ in $N(v)$ with label $i$. Moreover, $\beta(u_i)=b_i$ and there is a colourful cycle $C_i$ containing $u_i$ and $v$ that contains only forced vertices.
\end{lemma}
\begin{proof}
For each $i\in\{1,2,\ldots k-1\}$, by applying Lemma~\ref{lem:path} to $v$ and the set $\{i\}$, we get that there exists a decreasing path $vu_i$ where $\alpha(u_i)=i$ and $\beta(u_i)<\beta(v)=b_k$.
We shall choose $u_{k-1}$ to be $v_{k-1}$.
Because $u_i$ is adjacent to $v$ which is on the primary cycle, by Corollary~\ref{cor:colourfulcycle}, we know that $\beta(u_i)$ is a primary colour. 

We claim that $\beta(u_i)=b_i$ and that there is a colourful cycle containing $v$ and $u_i$ that contains only forced vertices. We shall use backward induction on $i$ prove this. Consider the base case when $i=k-1$. Since $u_{k-1}=v_{k-1}$, we know that $\beta(u_{k-1})=b_{k-1}$ and that there is a colourful cycle (the primary cycle) that contains $u_{k-1}$ and $v$ and also contains only forced vertices. Thus the claim is true for the base case. Let us assume that the claim has been proved for $u_{k-1},u_{k-2},\ldots,u_{i+1}$. If $\beta(u_i)=b_j>b_i$, then $b_j\in\{b_{i+1},b_{i+2},\ldots,b_{k-1}\}$ (recall that $\beta(u_i)$ is a primary colour). By the induction hypothesis, we know that the vertex $u_j\in N(v)$ has $\beta(u_j)=b_j$ and that there is a colourful cycle $C_j$ containing $u_j$ and $v$. Note that $u_j\neq u_i$ (as $\alpha(u_i)\neq\alpha(u_j)$), but $\beta(u_j)=\beta(u_i)=b_j$. Therefore, as $C_j$ contains $u_j$ and is a colourful cycle, it cannot contain $u_i$. Since $u_i$ is adjacent to $v$ which is on $C_j$, and $\beta(u_i)=b_j$, we now have a contradiction to Lemma~\ref{lem:colourfulcycle} (note that $u_jv$ is an edge of $C_j$ as every colourful cycle is a chordless cycle). So it has to be the case that $\beta(u_i)\leq b_i$.
By Lemma~\ref{lem:path}, there exists a path $y_iy_{i-1}y_{i-2}\ldots y_1$, where $y_i=u_i$, such that for $1\leq j\leq i-1$, $\alpha(y_j)=j$ and $\beta(y_j)<\beta(y_{j+1})$. Notice that $y_1y_2\ldots y_ivv_{k-1}\ldots v_{i+1}$ is a colourful path and therefore by Observation~\ref{obs:colpath}, $C_i=y_1y_2\ldots y_ivv_{k-1}\ldots v_{i+1}y_1$ is a colourful cycle containing both $u_i$ and $v$. Since $v_i$ is adjacent to $v_{i+1}$ which is on $C_i$, by Corollary~\ref{cor:colourfulcycle}, we know that there is some vertex $z$ on $C_i$ such that $\beta(z)=b_i$. Clearly, $z\in\{y_i,y_{i-1},\ldots,y_1\}$. Since $\beta(z)=b_i\geq\beta(y_i)$ (recall that $u_i=y_i$) and $y_iy_{i-1}\ldots y_1$ is a decreasing path, we have $z\notin\{y_{i-1},y_{i-2},\ldots,y_1\}$. Therefore, we have $z=y_i$, which implies that $\beta(u_i)=b_i$. Notice that each $y_j\in\{y_i,y_{i-1},\ldots,y_1\}$, because of the decreasing path $vy_iy_{i-1}\ldots y_j$, is a forced vertex. Thus, $C_i$ is a colourful cycle containing $u_i$ and $v$ that contains only forced vertices. This shows that for any $i\in\{1,2,\ldots,k-1\}$, $\beta(u_i)=b_i$ and there is a colourful cycle containing $v$ and $u_i$ that contains only forced vertices.

We shall now show that $u_i$ is the only vertex in $N(v)$ which has the label $i$. Suppose that there is a vertex $u\in N(v)$ such that $\alpha(u)=i$ and $u\neq u_i$. Since $u$ is adjacent to a colourful cycle containing only primary colours (the primary cycle), we can conclude from Corollary~\ref{cor:colourfulcycle} that $\beta(u)$ is a primary colour. Therefore, $\beta(u)=b_j$ for some $j\in\{1,2,\ldots,k-1\}$. From what we observed above, $\beta(u_j)=b_j$ and there exists a colourful cycle $C_j$ containing the vertices $v$ and $u_j$. Note that $u_j\neq u$ since if $j\neq i$, then $u_j$ and $u$ have different labels and if $j=i$, we know that $u_j\neq u$ (as we have assumed that $u_i\neq u$). Hence $u$ is not in $C_j$ (as $C_j$ already has a vertex $u_j$ with $\beta(u_j)=b_j$) but is adjacent to it. But now $C_j$ and $u$ contradict Lemma~\ref{lem:colourfulcycle} as $u_jv$ is an edge of $C_j$. Therefore, $u$ cannot exist.
\end{proof}

\begin{corollary}\label{cor:cyclewithv}
Let $C$ be any colourful cycle containing $v$. Then $C$ sees only primary colours.
\end{corollary}
\begin{proof}
Notice that from Lemma~\ref{lem:uis}, we know that for every primary colour $b_j\in\{b_1,b_2,\ldots,b_{k-1}\}$, there is a vertex $u_j$ with $\beta(u_j)=b_j$ that is adjacent to $v$. Because $v$ is in $C$, we can apply Corollary~\ref{cor:colourfulcycle} to $C$ and $u_j$ to conclude that $b_j$ is present in $C$. This means that every primary colour appears on at least one vertex of $C$. Since $C$ was a $k$-cycle, this means that $C$ sees only primary colours.
\end{proof}

The following corollary shows that $k\leq 6$. However, we do not need this fact to prove Theorem~\ref{thm:main}.

\begin{corollary}\label{cor:5or6}
$k$ is 5 or 6.
\end{corollary}
\begin{proof}
By Lemma~\ref{lem:uis}, there exists a vertex $u_2\in N(v)$ such that $\alpha(u_2)=2$ and $\beta(u_2)=b_2$. By applying Lemma~\ref{lem:path} to $u_2$ and the set $\{1\}$, we know that there exists $z\in N(u_2)$ such that $\alpha(z)=1$ and $\beta(z)<\beta(u_2)=b_2$. Now, the path $zu_2vv_{k-1}v_{k-2}\ldots v_3$ is a colourful path and hence by Observation~\ref{obs:colpath}, we have that $zv_3\in E(G)$. By just comparing labels, it is clear that $u_2\neq v_1$ and $z\neq v_2$. Further, $u_2\neq v_2$ as $vv_2\notin E(G)$ and $z\neq v_1$ as otherwise, there will be the triangle $vu_2zv$ in $G$. We then have the 6-cycle $vu_2zv_3v_2v_1v$ in $G$, which implies that $k=g(G)\leq 6$. Since $k>4$, we now have $k\in\{5,6\}$.
\end{proof}

\begin{lemma}\label{lem:primarycolours}
If $u\in V(G)$ is a forced vertex such that $\alpha(u)=i$, then $\beta(u)=b_i$. Moreover, if $P$ is any decreasing path from $v$ to $u$, then there is a colourful cycle which has $P$ as a subpath, contains only forced vertices, and sees exactly the primary colours.
\end{lemma}
\begin{proof}
Consider a forced vertex $u$. We shall prove the statement of the lemma for $u$ by backward induction on $\alpha(u)$. The statement is true for $\alpha(u)\in\{k,k-1\}$ as there is only one forced vertex each with labels $k$ and $k-1$---which are $v$ and $v_{k-1}$ respectively (recall that from Lemma~\ref{lem:uis}, $u_{k-1}=v_{k-1}$ is the only vertex in $N(v)$ with label $k-1$). Also, note that they are both in a colourful cycle (the primary cycle) that satisfies the required conditions. Let us assume that the statement of the lemma has been proved for $\alpha(u)\in\{k,k-1,\ldots,i+1\}$. Let us look at the case when $\alpha(u)=i$. Let $z$ be the predecessor of $u$ in the path $P$ and let $P_z$ be the subpath of $P$ that starts at $v$ and ends at $z$. Let $\alpha(z)=j$. By the induction hypothesis, $\beta(z)=b_j$ and $z$ is in a colourful cycle $C$ that contains only primary colours. By Corollary~\ref{cor:colourfulcycle}, we can infer that $\beta(u)$ is a primary colour. Since $P$ was a decreasing path, $\beta(u)\in\{b_1,b_2,\ldots,b_{j-1}\}$. If $\beta(u)=b_l$ with $b_j>b_l>b_i$, then notice that there already exists a neighbour $y$ of $z$ with $\alpha(y)=l$ and $\beta(y)<\beta(z)$, because the refined greedy algorithm set $\alpha(z)=j$. Note that $P_z\cup zy$ is a decreasing path from $v$ to $y$, which implies that $y$ is a forced vertex. Clearly, $u\neq y$ as $\alpha(u)\neq\alpha(y)$. Because of our induction hypothesis, $\beta(y)=b_l$ and there is a colourful cycle containing the path $P_z\cup zy$ as a subpath. As $\beta(u)=\beta(y)$, $u$ is outside this cycle but is a neighbour of $z$. This contradicts Lemma~\ref{lem:colourfulcycle}. Therefore, $\beta(u)\leq b_i$. Consider the decreasing path $y_iy_{i-1}\ldots y_1$ where $y_i=u$, and for $s\in\{1,2,\ldots,i-1\}$, $\alpha(y_s)=s$ and $\beta(y_s)<\beta(y_{s+1})$ which exists by Lemma~\ref{lem:path}. Again by Lemma~\ref{lem:path}, there exists a decreasing path $Q$ starting from $v$ whose vertices other than $v$ have exactly the labels in $\{i+1,i+2,\ldots,k\}$ that are not seen on $P_z$. By the induction hypothesis, we can now see that every colour in $\{b_{i+1},b_{i+2},\ldots,b_k\}$ occurs exactly once in the path $Q\cup P_z$. Since $y_iy_{i-1}\ldots y_1$ is a decreasing path in which every vertex has colour at most $b_i$, we can conclude that the path $P'=Q\cup P_z \cup zy_iy_{i-1}\ldots y_1$ is a colourful path. By Observation~\ref{obs:colpath}, the graph induced by $V(P')$ is a colourful cycle containing $v$, which we shall call $C'$. By Corollary~\ref{cor:cyclewithv}, we know that $C'$ contains only primary colours. Now, if $\beta(u)<b_i$, then because $uy_{i-1}\ldots y_1$ was a decreasing path, it should mean that $\beta(y_1)<b_1$, which is a contradiction. Thus, $\beta(u)=b_i$ and $C'$ is a cycle containing $P$ as a subpath and which contains only forced vertices and primary colours (note that each $y_s$, for $1\leq s\leq k-1$, is a forced vertex as there is the decreasing path $P_z\cup zy_iy_{i-1}\ldots y_s$ from $v$ to $y_s$).
\end{proof}

\begin{corollary}\label{cor:forced}
Let $u$ be a forced vertex with $\alpha(u)=i$. Then, for each $j\in\{1,2,\ldots,i-1\}$ there is exactly one forced vertex $u_j$ in $N(u)$ with label $j$. Moreover, for each $j\in\{1,2,\ldots,i-1\}$, there is a colourful cycle containing the edge $uu_j$, sees exactly the primary colours and contains only forced vertices.
\end{corollary}
\begin{proof}
As $u$ is a forced vertex, we have from Lemma~\ref{lem:primarycolours} that $\beta(u)=b_i$. We further know that there exists a decreasing path $P$ from $v$ to $u$. Lemma~\ref{lem:path} can be used to infer that there exists a vertex $u_j\in N(u)$ such that $\alpha(u_j)=j$ and $\beta(u_j)<\beta(u)$. As $P\cup uu_j$ is a decreasing path, $u_j$ is a forced vertex. Suppose for the sake of contradiction that there exists $u'_j\neq u_j$ such that $\alpha(u'_j)=j$ and $u'_j$ is a forced vertex in $N(u)$. By Lemma~\ref{lem:primarycolours}, we have that $\beta(u_j)=\beta(u'_j)=b_j$. Applying Lemma~\ref{lem:primarycolours} on $u_j$ and the decreasing path $P\cup uu_j$, we know that there exists a colourful cycle containing the edge $uu_j$, sees exactly the primary colours and contains only forced vertices. As $\beta(u_j)=\beta(u'_j)$, the vertex $u'_j$ is not on this cycle. This contradicts Lemma~\ref{lem:colourfulcycle}.
\end{proof}

It might be helpful to note that combining Corollaries~\ref{cor:5or6} and~\ref{cor:forced}, we get that the forced vertices in $G$ are as shown in Figure~\ref{fig:5} or Figure~\ref{fig:6}. But we do not use this observation for the proof.

\begin{figure}
	\renewcommand{\defradius}{0.2}
	\renewcommand{\vertexset}{(v5,4.5,4),(v4,5.5,3),(v3,7,1.5),(v2,7.75,0.75),(v1,8.5,0),(v35,4.0,3),(v25,3,3),(v14,5.0,2),(v24,5.5,1.5),(v13,3,0),(v12,4,1.5),(v23,4.5,2)}
	\renewcommand{\edgeset}{(v5,v4),(v4,v3),(v3,v2),(v2,v1),(v5,v1,,,2),(v13,v3),(v12,v24),(v14,v4),(v24,v4),(v25,v5),(v13,v25),(v23,v35),(v35,v5),(v14,v23),(v12,v35)}
	\begin{center}
		\begin{tikzpicture}
		\drawgraph
		\node at (\xy{v5}) {$v_5$};
		\node at (\xy{v4}) {$v_4$};
		\node at (\xy{v3}) {$v_3$};
		\node at (\xy{v2}) {$v_2$};
		\node at (\xy{v1}) {$v_1$};
		\node at (\xy{v13}) {$1$};
		\node at (\xy{v12}) {$1$};
		\node at (\xy{v14}) {$1$};
		\node at (\xy{v24}) {$2$};
		\node at (\xy{v23}) {$2$};
		\node at (\xy{v25}) {$2$};
		\node at (\xy{v35}) {$3$};
		\end{tikzpicture}
	\end{center}
	\caption{The forced vertices when $k=5$. The vertices in the primary cycle are named, while only the labels of the other forced vertices are shown.}\label{fig:5}
\end{figure}

\begin{figure}
	\renewcommand{\defradius}{0.2}
	\renewcommand{\vertexset}{(v6,1.5,7),(v5,3.5,5),(v4,5.5,3),(v3,6.9,1.6),(v2,7.7,0.8),(v1,8.5,0),(v64,1.5,6),(v63,-0.5,5),(v62,-1.5,5),(v653,4.45,3.3),(v652,3.75,3.75),(v651,3.25,4.25),(v643,2.5,5),(v642,1.5,5),(v641,0.5,5),(v632,-0.1,2.5),(v631,-0.5,1),(v621,-1.5,0),(v6541,-0.1,1.75),(v6542,5.5,2.25),(v6531,1.5,3.3),(v6532,0.5,2.5),(v6521,2.25,3.75),(v6432,2.75,4.25)}
	\renewcommand{\edgeset}{(v6,v5),(v5,v4),(v4,v3),(v3,v2),(v2,v1),(v6,v1,,,2),(v6,v64),(v64,v643),(v643,v6432),(v6432,v651),(v6,v64),(v64,v643),(v643,v6521),(v6521,v652),(v652,v5),(v5,v5),(v3,v621),(v621,v62),(v62,v6),(v4,v6542),(v6542,v631),(v63,v6),(v63,v632),(v4,v6541),(v6541,v632),(v632,v63),(v5,v653),(v653,v6532),(v6532,v641),(v641,v64),(v653,v6531),(v6531,v642),(v642,v64),(v5,v652),(v652,v6521),(v6521,v643),(v643,v64),(v5,v651),(v651,v6432),(v6432,v643),(v631,v63)}
	\begin{center}
		\begin{tikzpicture}
		\drawgraph
		\node at (\xy{v6}) {$v_6$};
		\node at (\xy{v5}) {$v_5$};
		\node at (\xy{v4}) {$v_4$};
		\node at (\xy{v3}) {$v_3$};
		\node at (\xy{v2}) {$v_2$};
		\node at (\xy{v1}) {$v_1$};
		\node at (\xy{v64}) {$4$};
		\node at (\xy{v63}) {$3$};
		\node at (\xy{v62}) {$2$};
		\node at (\xy{v653}) {$3$};
		\node at (\xy{v652}) {$2$};
		\node at (\xy{v651}) {$1$};
		\node at (\xy{v653}) {$3$};
		\node at (\xy{v643}) {$3$};
		\node at (\xy{v642}) {$2$};
		\node at (\xy{v641}) {$1$};
		\node at (\xy{v642}) {$2$};
		\node at (\xy{v632}) {$2$};
		\node at (\xy{v631}) {$1$};
		\node at (\xy{v621}) {$1$};
		\node at (\xy{v6541}) {$1$};
		\node at (\xy{v6542}) {$2$};
		\node at (\xy{v6531}) {$1$};
		\node at (\xy{v6532}) {$2$};
		\node at (\xy{v6521}) {$1$};
		\node at (\xy{v6432}) {$2$};
		
		\end{tikzpicture}
	\end{center}
	\caption{The forced vertices when $k=6$. The vertices in the primary cycle are named, while only the labels of the other forced vertices are shown.}\label{fig:6}
\end{figure}

Clearly, the only forced vertex with label $k$ is $v_k$. And by Lemma~\ref{lem:uis}, we also have that the only forced vertex with label $k-1$ is $v_{k-1}$. This gives us the following observation.

\begin{observation}\label{obs:final}
There are exactly $k-2$ forced vertices in $N(v_{k-2})$.
\end{observation}
\begin{proof}
From Corollary~\ref{cor:forced}, we know that for each $i\in\{1,2,\ldots,k-3\}$, there is exactly one forced vertex with label $i$ in $N(v_{k-2})$. The only forced vertices with labels higher than $k-2$ are $v_{k-1}$ and $v_k$. Since $v_{k-2}v_k\notin E(G)$ and $v_{k-2}v_{k-1}\in E(G)$, there are exactly $k-2$ forced vertices in $N(v_{k-2})$.
\end{proof}

\begin{lemma}\label{lem:final}
Every vertex in $N(v_{k-2})$ is a forced vertex.
\end{lemma}
\begin{proof}
Suppose for the sake of contradiction that there exists a vertex $u\in N(v_{k-2})$ that is not a forced vertex. Applying Lemma~\ref{lem:colourfulcycle} on the primary cycle and $u$, we get that $\beta(u)$ is a primary colour other than $b_{k-1},b_{k-2},b_{k-3}$. By applying Corollary~\ref{cor:forced} to $v_{k-2}$, we know that for each $j\in\{1,2,\ldots,k-3\}$, there is a forced vertex $u_j\in N(v_{k-2})$ such that $\alpha(u_j)=j$ and there is a colourful cycle $C_j$ containing the edge $v_{k-2}u_j$, sees exactly the primary colours and contains only forced vertices. Note that by Lemma~\ref{lem:primarycolours}, we have $\beta(u_j)=b_j$. Then by applying Lemma~\ref{lem:colourfulcycle} to $C_j$ and $u$, for each $j\in\{1,2,\ldots,k-3\}$, we can conclude that $\beta(u)\notin\{b_1,b_2,\ldots,b_{k-3}\}$. Further, by applying Lemma~\ref{lem:colourfulcycle} to the primary cycle and $u$, we get that $\beta(u)\neq b_{k-1}$. Therefore, $\beta(u)=b_k$. By Lemma~\ref{lem:path} applied on the vertex $v_{k-1}$ and set $\{b_{k-3},b_{k-4},\ldots,b_1\}$, there exists a decreasing path $v_{k-1}v'_{k-3}v'_{k-4}\ldots v'_1$ where $\alpha(v'_i)=i$. Each $v'_i$, for $1\leq i\leq k-3$, is a forced vertex since $v_kv_{k-1}v'_{k-3}v'_{k-4}\ldots v'_i$ is a decreasing path from $v_k$ to $v'_i$. From Lemma~\ref{lem:primarycolours}, we now have that for each $1\leq i\leq k-3$, $\beta(v'_i)=b_i$. Again applying Lemma~\ref{lem:path} to $v_{k-1}$ and the set $\{k-4,k-5,\ldots,1\}$, we get a decreasing path $v_{k-1}u'_{k-4}u'_{k-5}\ldots u'_1$,  for which using similar arguments as before, we can see that $u'_i$, for $1\leq i\leq k-4$, is a forced vertex with $\beta(u'_i)=b_i$. Now applying Lemma~\ref{lem:path} to $v_k$ and the set $\{k-2,k-3\}$, we get a decreasing path $v_kx_{k-2}x_{k-3}$, and using similar arguments as before we get that $x_i$, for $i\in\{k-2,k-3\}$, is a forced vertex with $\beta(x_i)=b_i$. Note that $uv_{k-2}v_{k-1}v'_{k-3}v'_{k-4}\ldots v'_1$ is a colourful path, implying that $uv'_1\in E(G)$. Also, $u'_1u'_2\ldots u'_{k-4}v_{k-1}v_kx_{k-2}x_{k-3}$ is a colourful path, and hence $u'_1x_{k-3}\in E(G)$. Now, $uv_{k-2}v_{k-1}u'_{k-4}u'_{k-5}\ldots u'_1x_{k-3}$ is a colourful path, which implies that $ux_{k-3}\in E(G)$. Further, $v'_1v'_2\ldots v'_{k-3}v_{k-1}v_kx_{k-2}$ is a colourful path, which gives us $v'_1 x_{k-2}\in E(G)$. Therefore, we have the 4-cycle $uv'_1x_{k-2}x_{k-3}u$ in $G$, which is a contradiction.
\end{proof}

From Observation~\ref{obs:final} and Lemma~\ref{lem:final}, we have that $v_{k-2}$ has exactly $k-2$ neighbours, which is a contradiction to Observation~\ref{obs:mindegree}. This completes the proof of Theorem~\ref{thm:main}.

\section{Conclusion}
The results of this paper imply that for any properly coloured graph $G$ with $g(G)\geq\chi(G)>3$, there exists an induced colourful path on $\chi(G)$ vertices in $G$. The question of whether every properly coloured graph $G$ contains an induced colourful path on $\chi(G)$ vertices remains open for the case $3<g(G)<\chi(G)$.

\section*{Acknowledgements}
The authors would like to thank an anonymous referee whose suggestions led to considerable simplification of the proof. A part of this work was done when Manu Basavaraju was a postdoc at the Max-Planck-Institut f\"ur Informatik, Saarbr\"ucken. L.~Sunil Chandran was supported by a fellowship from the Humboldt Research Foundation. Mathew~C. Francis would like to acknowledge the support provided by the DST-INSPIRE Faculty Award IFA12-ENG-21.
\bibliography{references}
\end{document}